\documentclass{amsart}
\usepackage{amssymb,latexsym}
\usepackage[mathscr]{eucal}
\usepackage{mathtools}
\usepackage{enumitem}

\theoremstyle{plain}
\newtheorem{theorem}{Theorem}
\newtheorem{corollary}[theorem]{Corollary}
\newtheorem{conjecture}[theorem]{Conjecture}
\newtheorem{lemma}[theorem]{Lemma}

\theoremstyle{definition}
\newtheorem{example}{Example}
\newtheorem{definition}{Definition}

\newcommand{\refT}[1]{Theorem~\ref{T#1}}
\newcommand{\refC}[1]{Corollary~\ref{C#1}}
\newcommand{\refCon}[1]{Conjecture~\ref{Con}}
\newcommand{\refEx}[1]{Example~\ref{Ex#1}}
\newcommand{\refL}[1]{Lemma~\ref{L#1}}

\newcommand{\gd}{\delta}  \newcommand{\ep}{\varepsilon}  
\newcommand{\gf}{\varphi}  \newcommand{\vta}{\vartheta}

\newcommand{\Z}{\mathbb{Z}}  \newcommand{\N}{\mathbb{N}}
\newcommand{\CA}{\mathcal{A}}    
\newcommand{\CN}{\mathcal{N}}  \newcommand{\CM}{\mathcal{M}}
\newcommand{\CR}{\mathcal{R}}
\newcommand{\ut}{\underline{t}}  \newcommand{\us}{\underline{s}}

\newcommand{\cycp}{\Phi_{pqr}}
\newcommand{\la}{\langle}  \newcommand{\ra}{\rangle}
\newcommand{\sq}{^{}_Q}  \newcommand{\sr}{^{}_R}
\newcommand{\sel}{^{}_L}  \newcommand{\slp}{^{}_{L'}}

\DeclareMathOperator{\sgn}{sgn}

\author[G. Bachman]{Gennady Bachman}
\address{Department of Mathematical Sciences\\ University of Nevada Las Vegas\\
4505 Maryland Parkway, Box 454020\\
Las Vegas, Nevada 89154-4020, USA}
\email{gennady.bachman@unlv.edu}

\date{January 2026}

\begin{document}

\title[Cyclotomic Polys: Heights and Diameters]{On Heights and Diameters of Ternary Cyclotomic and Inclusion-Exclusion Polynomials}

\begin{abstract}
For the $n$th cyclotomic polynomial $\Phi_n$, let $A(n)$ denote the greatest absolute value of its coefficients, its height, and let $D(n)$ denote the difference between its largest and smallest coefficients, its diameter. We show that for any odd prime $p$ and an integer $h$ in the range $1\le h\le(p+1)/2$, there are arbitrarily large primes $q$ and $r$ such that $\cycp$ has the height $h$. This certainly answers the question of whether every natural number occurs as the height of some cyclotomic polynomial. Our construction specifies explicit choices of $q$ and $r$ with $A(pqr)=h$, and for these choices $D(pqr)$ has one of two values: it is either $2h$ or $2h-1$, depending on the congruence class of $h$ modulo $p$.
\end{abstract}



\maketitle

\section{Introduction}
Let $\Phi_n$ be the $n$th cyclotomic polynomial, that is, $\Phi_n(x)=\prod(x-\zeta)$, where the product is taken over the primitive $n$th roots of unity $\zeta$, and write
\[
\Phi_n(x)=\sum_{0\le m\le\gf(n)}a_mx^m \qquad [a_m=a(m;n)],
\]
where $\gf$ is the Euler's totient function. It is well known that the coefficients $a_m$ are integral and there is a growing body of literature studying their properties and especially properties of the function
\[
A(n)\coloneqq\max_m|a_m|=\max_m|a(m;n)|.
\]
(See \cite{Bz1,Bz2} and the references therein.) As has become increasingly common in the literature, we shall refer to $A(n)$ as the height of $\Phi_n$. It is an interesting and nontrivial question if there is a polynomial $\Phi_n$ of any given height $h$? To gain a better understanding of this question we review the following easy facts (see, for example, \cite{Le}). To investigate possible heights, it suffices to study possible values of $A(\prod_{1\le i\le k}p_i)$, for primes $2<p_1<\dots<p_k$. This is because if $\prod_{1\le i\le k}p_i$ is the ``odd squarefree kernel'' of $n$, then $A(n)=A(\prod_{1\le i\le k}p_i)$. If $k$ is the number of distinct odd prime factors of $n$, polynomial $\Phi_n$ is said to be of order $k$. The height question naturally splits into the family of equations
\begin{equation}\label{1}
A\Bigl(\prod_{1\le i\le k}p_i\Bigr)=h,
\end{equation}
one equation for each order $k$. It is an easy fact that $A(\prod_{1\le i\le k}p_i)\equiv1$, for $k<3$. For $k\ge3$ the problem is challenging and wide open.

Recent progress in the study of ternary cyclotomic polynomials ($k=3$) furnished tools that allowed one to tackle the solubility of
\begin{equation}\label{1t}
A(pqr)=h
\end{equation}
in odd primes $p<q<r$. Early successes were limited to some special cases where $h$ is small and are as follows. 
\begin{itemize}
\item Bachman \cite{Ba.Flat}: For every prime $p$, $A(pqr)=1$ is soluble in arbitrary large primes $q$ and $r$.
\item Elder\cite[unpublished]{El}/Zhang\cite{Zh1}: For every prime $p$, $A(pqr)=2$ is soluble in arbitrary large primes $q$ and $r$.
\item Zhang\cite{Zh2}: For every prime $p\equiv1\pmod3$, $A(pqr)=3$ is soluble in arbitrary large primes $q$ and $r$.
\end{itemize}
The latter result of Zhang was preceded by a conditional result of Gallot, Moree and Wilms \cite{GMW}: For every prime $p$ such that $2p-1$ is also prime, $A(p(2p-1)r)=3$ is soluble in arbitrary large primes $r$. In view of the history of this problem, the case $h=1$ is particularly intriguing (see discussion in \cite{Ba.Flat}) and also the easiest (unless one treats $h$ in \eqref{1t} as a function of $p$ and takes $h=(p+1)/2$) and certainly the best understood case to date---see \cite{Ka.Flat, Ka.4Flat, El, BM.IEP}. A particularly appealing result in this case is due to Kaplan \cite{Ka.Flat} and is stated as \refT{4} in the next section.

For an arbitrary $h$, a conditional affirmative answer to \eqref{1t} was recently given by Kosyak, Moree, Sofos and Zhang \cite{KMSZ}. They showed that if a certain short interval $I_h$ depending on $h$ contains a prime $p$, roughly
\begin{equation}\label{2}
2h-1-\sqrt{2h-1}<p\le2h-1,
\end{equation}
then there are arbitrarily large primes $q$ and $r$ such that $A(pqr)=h$. Of course, the difficulty here is that the existence of a prime in a short interval \eqref{2} remains an open problem---it should be compared with Legendre's conjecture that there is always a prime between consecutive squares. However, using a recent advance on the distribution of primes due to Heath-Brown \cite{HB}, they are able to say that $A(pqr)=h$ is soluble for almost all $h$. More precisely, they conclude that the number of ``bad'' $h\le x$ is $O_\ep(x^{3/5+\ep})$, as $x$ goes to infinity.

Using a different approach, Bachman, Bao and Wu \cite{BBW.Involve} showed that for every positive integer $h$ the equation
\[
A(pqr)=h\quad\text{or}\quad h+1
\]
is always soluble in arbitrarily large primes $q$ and $r$. A notable feature of their approach is that it has only minimal requirements from the theory of distribution of primes, namely, the Dirichlet's theorem for primes in arithmetic progressions. They also make the following conjecture.

\begin{conjecture}\label{Con}For an odd prime $p$, let $M(p)\coloneqq\max_{q,r}A(pqr)$. Then for every $p$ and $1\le h\le M(p)$, the equation $A(pqr)=h$ is soluble in arbitrarily large primes $q$ and $r$.
\end{conjecture}

Establishing the solubility of \eqref{1t} for all $h\ge1$ was the impetus for the current work. Our analysis yields a number of new results for heights and diameters of ternary cyclotomic and inclusion-exclusion polynomials. Those will be given in the next section after the necessary background material has been presented. But we include in this section the simplest and cleanest of our results, a result which may be thought of as a step towards \refCon{Con}.

\begin{theorem}\label{T2}
Let $p$ be an odd prime and let $h$ be any integer in the range $1\le h\le(p+1)/2$. Then there exist arbitrarily large primes $q$ and $r$ such that $A(pqr)=h$.
\end{theorem}

We conclude this section with a quick ward about the equation \eqref{1} for $k\ge4$. Aside from explicitly calculating $A(n)$ for specific values of $n$, almost nothing is known. A particularly appealing old problem is to determine whether there are polynomials $\Phi_n$ of arbitrarily high order and height 1. Polynomials of height 1 are said to be flat. Despite extensive computations, no flat polynomials of order 5 or higher have ever been found. Kaplan \cite{Ka.4Flat} proved the following general result. Let $n$ be arbitrary and let primes $r$ and $R$ satisfy $R>r>n$ and $R\equiv r\pmod n$. Then $A(nR)=A(nr)$. Combining this result with computationally obtained evaluation $A(3\cdot5\cdot31\cdot929)=1$, he showed that there are infinitely many flat polynomials of the form $\Phi_{3\cdot5\cdot31\cdot R}$.

\section{Background and Results}
Modifying our earlier convention, we now let $p$, $q$ and $r$ denote three relatively prime in pairs positive integers $\ge3$ (so they need not be primes). Put $T=\{p,q,r\}$ and
\begin{equation}\label{3}
Q_T(x)=Q_{\{p,q,r\}}(x)\coloneqq
\frac{(x^{pqr}-1)(x^p-1)(x^q-1)(x^r-1)}{(x^{pq}-1)(x^{qr}-1)(x^{rp}-1)(x-1)}.
\end{equation}
Routine application of the inclusion-exclusion principle to the roots of the polynomials on the right side shows that $Q_T$ reduces to a polynomial, a \emph{ternary inclusion-exclusion polynomial}. We refer the reader to \cite{Ba.IEP} for an introduction to inclusion-exclusion polynomials. This class of polynomials may be thought of as a tool for studying coefficients of cyclotomic polynomials and other divisors of $x^n-1$, or as looking at cyclotomic polynomials from a combinatorial rather than algebraic point of view. In particular, when the parameters $p$, $q$ and $r$ are distinct odd primes, the polynomial $Q_{\{p,q,r\}}$ is better known as the cyclotomic polynomial $\cycp$.

Polynomial $Q_T$ is of degree
\begin{equation}\label{4}
\gf(T)=\gf(p,q,r)\coloneqq(p-1)(q-1)(r-1)
\end{equation}
and we write, as before,
\[
Q_T(x)=\sum_{0\le m\le\gf(T)}a_mx^m \qquad[a_m=a(m;T)].
\]
We also put
\[
\CA_T\coloneqq\{a_m\} \quad\text{and}\quad A(T)\coloneqq\max_m|a_m|,
\]
as well as
\[
A^+(T)\coloneqq\max_m a_m \quad\text{and}\quad A^-(T)\coloneqq\min_m a_m.
\]
Quantities $A^{\pm}(T)$ completely determine $\CA_T$ for it is known (see \cite{GM.ConsecCoeffs} and \cite{Ba.IEP}) that
\begin{equation}\label{5}
\CA_T=[A^-(T),A^+(T)]\cap\mathbb Z.
\end{equation}

We shall see that the following assumptions on triples $T$ allow us to evaluate $A^{\pm}(T)$ with a reasonable amount of effort. Let $p\ge3$ be an arbitrary fixed integer and let $t\in[1,p-1]$ be relatively prime to $p$. For each such pair of $p$ and $t$, let $q$, $r$ and $r'$ be any integers satisfying the conditions
\begin{equation}\label{6}
r,r'>q>p^2,\quad q\equiv t\pmod p,\quad rt\equiv1\pmod{pq},
\end{equation}
and
\begin{equation}\label{7}
r't\equiv-1\pmod{pq}.
\end{equation}
Note that if $t\ne1$, then any positive $n$ satisfying $nt\equiv\pm1\pmod{pq}$ is automatically $>q$, so our assumption $r,r'>q$ was only needed to preclude the possibility $r=t=1$ in \eqref{6}. Note also that the triples $T=\{p,q,r\}$ and $T'=\{p,q,r'\}$ are each relatively prime in pairs. Evaluation of $A^{\pm}$ for such triples depends on the parameters $\ut$ and $\us$ defined as follows. First, let $s\in[1,p-1]$ be the inverse of $t$ modulo $p$, so that $r\equiv s\pmod p$. Now put
\[
\ut\coloneqq\min(t,p-t) \quad\text{and}\quad \us\coloneq\min(s,p-s).
\]
We are now ready to state the result that underpins all the findings in our paper.

\begin{theorem}\label{T3}
Let $T=\{p,q,r\}$ and $T'=\{p,q,r'\}$ be any triples satisfying \eqref{6} and \eqref{7}. Then the sets $\CA_T$ and $\CA_{T'}$ satisfy the identity
\begin{equation}\label{8}
\CA_{T'}=-\CA_T.
\end{equation}
Therefore, by \eqref{5}, we only need to evaluate $A^\pm(T)$. These are as follows.
\begin{enumerate}[leftmargin=*, label=\emph{(\roman*)}]
\item If $t=s=1$, then
\[
A^-(T)=-1, \quad A^+(T)=1, \quad A(T)=1.
\]
\item If $t>1$ and $\us<\ut$, then
\[
A^-(T)=-\us-1, \quad A^+(T)=\us+1, \quad A(T)=\us+1.
\]
\item If $t>1$, $\us\ge\ut$ and $(\ut,\us)=(t,s)$ or $(p-t,p-s)$, then
\[
A^-(T)=-\us, \quad A^+(T)=\us+1, \quad A(T)=\us+1.
\]
\item If $t>1$ $\us\ge\ut$ and $(\ut,\us)=(t,p-s)$ or $(p-t,s)$, then
\[
A^-(T)=-\us-1, \quad A^+(T)=\us, \quad A(T)=\us+1.
\]
\end{enumerate}
\end{theorem}

\begin{example}\label{Ex1}
Recall that cyclotomic polynomials of order $<3$ are all flat. The first non-flat cyclotomic polynomial is $\Phi_{105}$, and it is a classic example that its coefficients form a set $\CA_{\{3,5,7\}}=\{-2,-1,0,1\}$. It is interesting to note that \refT{3} gives a theoretic evaluation of $\CA_{\{3,5,7\}}$. To see this we take
\[
p=3,\quad t=2,\quad s=2,\quad q=5,\quad r'=7,
\]
and ingnore the requirement $q>p^2$. Then
\[
(\ut,\us)=(1,1)=(p-t,p-s)
\]
and \refT{3}(iii) gives
\[
A^-(3,5,r)=-1\quad\text{and}\quad A^+(3,5,r)=2,
\]
from which the conclusion follows by \eqref{8} and \eqref{5}.
\end{example}

This is a good place to make a brief comment on our requirement $q>p^2$ in \refT{3}. This condition will be seen to yield technical simplifications in our proof of the theorem. We did not explore the extent to which this requirement can be relaxed, and it is entirely possible that simply having $q>p$ is not sufficient for our conclusions. But it happens to be sufficient in this example.

Part (i) of \refT{3} is not new and in certain important ways is different from the other parts, as we now explain. In \cite{Ba.Flat}, the author gives the result
\[
q\equiv-1\pmod p \quad\text{and}\quad r\equiv1\pmod{pq} \implies A(p,q,r)=1.
\]
But in fact, the method of that paper reaches the same conclusion under any of the four assumptions $q\equiv\pm1\pmod p$ and $r\equiv\pm1\pmod{pq}$ (containing \refT{3}(i)). Then Kaplan \cite{Ka.Flat} showed that the condition $r\equiv\pm1\pmod{pq}$ alone is sufficient for this.

\begin{theorem}\label{T4}
$A(p,q,r)=1$, if $r\equiv\pm1\pmod{pq}$.
\end{theorem}

Later, Bachman and Moree \cite{BM.IEP} studied heights of polynomials $Q_{\{p,q,r\}}$, where $p<q<r$ and $r$ is of the form $r=kpq\pm r_0$, $1\le r_0<p$ (prime to $pq$). The method developed in that work reduces the height $A(p,q,r)$ to that of $A(p,q,r_0)$ and, in particular, gives a new proof of \refT{4}. And the point that we are making now is that:
\begin{enumerate}[leftmargin=*]
\item[a)] The subclass of ternary inclusion-exclusion polynomials studied in that paper and the subclass analysed in \refT{3} intersect precisely when $r_0=1$ and $t=s=1$; and
\item[b)] The method of Bachman and Moree is really the preferred way of looking at these polynomials.
\end{enumerate}
At any rate, one readily verifies that \refT{3}(i) is contained in \refT{4} and we choose not to include here its new proof which follows the method developed here.

Strictly speaking, the forgoing discussion of earlier work related to \refT{4} is somewhat inaccurate and requires clarification. Kaplan's proof of \refT{4} as well as related work of Bachman preceding it, were carried out only in the context of cyclotomic polynomials. Their arguments, however, readily carry over to and remain valid for the entire class of inclusion-exclusion polynomials. In the subsequent discussion, we will encounter a number of other situations which are analogous to the situation we just described. To avoid repetition, we will no longer mention this point and simply state a result as valid for the inclusion-exclusion polynomials if we know it to be so. 

Next, we record all the cases of the equation $A(p,q,r)=h$ that can be solved using \refT{3}.

\begin{corollary}\label{C5}
Fix $p\ge3$. The equation $A(p,q,r)=h$ certainly has solutions in arbitrarily large $q$ and $r$, which may be taken to be prime, for $h=1$ and $h$ satisfying $\gcd(h-1,p)=1$ in the range:
\begin{itemize}
\item $2\le h\le(p+1)/2$, if $p$ is odd, and
\item $h=2k$, $1\le k\le l$, if $p=4l$ or $4l+2$.
\end{itemize}
\end{corollary}

For an odd prime $p$, \refC{5} implies the solubility of $A(p,q,r)=h$ in arbitrarily large primes $q$ and $r$, for all $1\le h\le(p+1)/2$. This principal special case of \refC{5} was stated as \refT{2} in the introduction.

The diameter of a polynomial $Q_T$ is defined to be
\[
D(T)=D(p,q,r)\coloneqq A^+(p,q,r)-A^-(p,q,r).
\]
It is known that
\begin{equation}\label{9}
2\le D(T)\le p\qquad[p=\min(p,q,r)].
\end{equation}
The lower bound  is trivial for it is very easy to see that every $Q_T$ has positive and negative coefficients. $D(T)=2$ if $Q_T$ is flat. The upper bound was given in \cite[Corollary 3]{Ba.Start}. This bound is also sharp and examples of cyclotomic polynomials satisfying $D(p,q,r)=p$ were given in \cite{Ba.OLSC} and then by Moree and Ro\c{s}u in \cite{MR}. It is interesting to point out that in both of these papers the evaluation of heights and diameters were indirect and ran as follows. For a fixed prime $p$, it was shown how to choose (large) primes $q$ and $r$ such that it was possible to identify a specific ``small''coefficient with value $-\frac{p-1}2+k$ and a specific ``large'' coefficient with value $\frac{p+1}2+k$. In \cite{Ba.OLSC} this was done only for $k=0$, and in \cite{MR} this was done for every $k$ in the range
\[
0\le k\le\frac14(\sqrt{4p-11}-1).
\]
Then the authors observed that the difference of the large and small coefficients was $=p$. It automatically followed, by
\eqref{9}, that the height must be $\frac{p+1}2+k$ and the perimeter $p$. Notice also that this result of Moree and Ro\c{s}u shows that $M(p)>\frac{p+1}2$ for all primes $p\ge11$.

Findings of \refT{3} yield new results for the diameter problem. For the sake of simplicity, let us eliminate the requirement $\gcd(h-1,p)=1$ in \refC{5} by assuming that $p$ is an odd prime. In fact, we will state our diameter results in the language of cyclotomic polynomials $\cycp$ and write $D(pqr)$ in place of $D(p,q,r)$. Thus, we are now concerned with solving the equation
\begin{equation}\label{10}
D(pqr)=d \qquad[2\le d\le p,\ r>q>p].
\end{equation}
\begin{example}\label{Ex2}
Let us begin with four examples, starting with two evaluations of $D(pqr)$ that were previously known.
\begin{enumerate}[leftmargin=*]
\item[a)] \refT{3}(i) solves \eqref{10} with $d=2$, for each (prime) $p$.
\item[b)] \refT{3}(iv) with $t=2$ and $s=\frac{p+1}2$, so that $(\ut,\us)=(2,\frac{p-1}2)$, solves \eqref{10} with $d=p$, for each $p$.
\item[c)] \refT{3}(iii) with $t=s=p-1$, so that $\ut=\us=1$, solves \eqref{10} with $d=3$, for each $p$.
\item[d)] It is not possible to use \refT{3} to solve \eqref{10} with $d=4$.
\end{enumerate}
\end{example}
We should also remark that the existence of arbitrarily large primes $q$ and $r$ satisfying the required conditions \eqref{6} is guaranteed by the Dirichlet's theorem for primes in arithmetic progressions.

With $p$ fixed, $D(pqr)$ takes on at most $p-1$ possible values, by \eqref{9}. \refT{3} is sufficient to show that $D(pqr)$ can assume at least half of these values.

\begin{corollary}\label{C6}
With $p$ fixed, the equation $D(pqr)=d$ is soluble in arbitrary large primes $q$ and $r$ for at least $(p+1)/2$ different values of $2\le d\le p$.
\end{corollary}

We saw that \refT{3} yields no information about polynomials with diameter 4, but shows that $D(pqr)=2,3,p$ are always possible, for a fixed $p$. In fact, we can do a little better than this.

\begin{corollary}\label{C7}
Fix $p$. There are arbitrarily large $q$ and $r$ such that $D(pqr)=2$, and $D(pqr)=2h$, for each $3\le h<\sqrt p$. Moreover, there are arbitrarily large $q$ and $r$ such that $D(pqr)=p$, and $D(pqr)=p-2k$, for each $1\le k<\sqrt p/2-1$.
\end{corollary}

So \refC{7} shows that every even number $\ge6$ is a diameter of some $\cycp$ with $p$ sufficiently large. We do not have an analogue of this for odd numbers $\ge5$. But at least we can assert the following weaker conclusion.
\begin{corollary}\label{C8}
For each odd number $d\ge3$, the equation $D(pqr)=d$ is soluble for at least one prime $p$.
\end{corollary}

We give proofs of all the corollaries first and collect them in the next section. The rest of the paper gives a proof of \refT{3}, which is elementary but rather lengthy.

\section{Proofs of Corollaries}
\subsection{Proof of \refC{5}}
\refT{4} takes care of $h=1$. If $h-1<p/2$ and $(h-1,p)=1$, let $1\le t<p$ be such that $t(h-1)\equiv1\pmod p$. Now take $q$ and $r$ satisfying the condition \eqref{6}. Then $\us=s=h-1$ and, by \refT{3}, $A(p,q,r)=h$. Note that, by the Dirichlet's theorem for primes in arithmetic progressions, $q$ and $r$ may be taken to be prime and arbitrarily large.

It remains to verify the upper bounds in the range of $h$. Observe that if $p$ is odd, then $(\frac{p-1}2,p)=(p-1,p)=1$ and we may take $h-1$ above as large as $\frac{p-1}2$. If $p$ is even, say $p=4l$ or $p=4l+2$, then $(2l-1,p)=1$ and we may take $h-1$ as large as $2l-1$ in this case.

\subsection{Proof of \refC{6}}
We begin by considering special cases of \refT{3}. The first special case corresponds to the pair $\ut=\us=1$. This pair corresponds to the two values of $t$, 1 and $p-1$, yielding the diameters 2 and 3, as discussed in \refEx{2}(a,c). The other special case occurs only for primes $p\equiv1\pmod 4$ and corresponds to the fact that $-1$ is a quadratic residue for such primes. So let $u$ be the unique integer in $[2,\frac{p-1}2]$ such that $u^2\equiv -1\pmod p$. Taking $t=u$ or $t=p-u$ results in the pair $\ut=\us=u$ and, by \refT{3}(iii,iv), either of these choices of $t$ yields the diameter $2u+1$.

In the remaining cases $\ut\neq\us$ and \refT{3} yields the evaluation
\begin{equation}\label{3.1}
D(pqr)=\begin{cases} 2\us+2, &\text{if } \us<\ut \\
2\us+1, &\text{if } \us>\ut,  \end{cases}
\end{equation}
for primes $q$ and $r$ satisfying \eqref{6}. Now, for each $a$ in the range $2\le a\le\frac{p-1}2$ there corresponds a unique $b$, $2\le b\le\frac{p-1}2$, such that $ab\equiv\pm1\pmod p$. It follows that the set of integers $n\neq u$ in the range $2\le n\le\frac{p-1}2$ partitions uniquely into pairs
\begin{equation}\label{3.2}
\Bigl\{\, a,b \mid 2\le a<b\le\frac{p-1}2 \text{ and } ab\equiv\pm1\pmod p \,\Bigr\}.
\end{equation}
For each such pair $\{a,b\}$, taking $t=b$, so that $\ut=b$ and $\us=a$, and taking $t=a$, so that $\ut=a$ and $\us=b$, yields two different values of $D$ in \eqref{3.1}, namely $2a+2$ and $2b+1$. Plainly, these values do not repeat for different choices of $\{a,b\}$.

The proof is completed by counting the distinct values of $D$ generated above. One readily verifies that in both cases $p\equiv\pm1\pmod 4$ the count yields the value $\frac{p+1}2$.

\subsection{Proof of \refC{7}}
We address the even diameters first. Diameters 2 and 4 have already been discussed in \refEx{2}(a,d). For larger even diameters, we refer to the proof of \refC{6} and consider pairs $\{a,b\}$ given in \eqref{3.2}. We have seen that taking $t=b$ yields the diameter $2(a+1)$. The first claim now follows on observing that, for $2\le a<\sqrt p-1$, we have
\[  ab\equiv\pm1\pmod p \implies ab\ge p-1 \implies b>a.  \]

We now address the odd diameter claim. In \refEx{2}(b) we established the solubility of $D(pqr)=p$ for each $p$. It remains to consider $D(pqr)=p-2k$ for $k<\sqrt p/2-1$ and $p\ge17$. We argue in essentially the same way as in the even case. This time, for each pair $\{a,b\}$ in \eqref{3.2}, we want to take $t=a$ to get the diameter $2b+1=p-2k$. We claim that this certainly works for every $b=\frac{p-1}2-k$, with $0\le k<\sqrt p/2-1$. To verify this we need to show that if
\[  \Bigl(\frac{p-1}2-l\Bigr)\Bigl(\frac{p-1}2-k\Bigr)\equiv\pm1\pmod p,  \]
for $k$ in this range and $l$ in $0\le l<p/2$, then $l\ge k$, so that $a=\frac{p-1}2-l\le b$. This congruence is equivalent to the congruence
\[  (2l+1)(2k+1)\equiv\pm4\pmod p \implies (2l+1)(2k+1)\ge p-4.  \]
But for $l<k<\sqrt p/2-1$, we have
\[  (2l+1)(2k+1)\le4k^2-1<p-4\sqrt p+3<p-4,  \]
and the claim follows.

\subsection{Proof of \refC{8}}
We already know that for each $p$, there are $q$ and $r$ such that $D(pqr)=p$. So to verify our claim it remains to consider (composite) numbers $d=2b+1$, with $b\ge4$. Now, we have seen in the proof of \refC{6} that to get $D(pqr)=2b+1$ using \refT{3} we must show that, given $b$, there is a prime $p\ge2b+1$ and an integer $a\le b$ such that $ab\equiv\pm1\pmod p$. The existence of $p$ and $a$ satisfying these requirements follows from a general result of Laishram and Shorey \cite[Theorem 1]{LS}. Indeed, for our purposes the following fact contained in their work will suffice.

\begin{lemma}\label{L1}
Let $P(n)$ denote the largest prime factor of a natural number $n$, and put $F(n)=\prod_{i=1}^n(1+in)$. Then, for all $n\ge4$, we have
\[  P(F(n))>2n+1.  \]
\end{lemma}

Applying \refL{1} with $n=b$ shows that there is a prime $p\ge2b+1$ and an integer $2\le a\le b$ such that $ab\equiv-1\pmod p$, as required.

\section{Preliminaries}
In this section we collect the required background material on $Q_T$ of general nature---we do not yet specialize to require the parameters $p$, $q$ and $r$ to satisfy \eqref{6}. Our starting point for studying coefficients of $Q_T$ is the relation
\begin{equation}\begin{aligned}\label{4.1}
Q_T\equiv(1-x^q-x^r+x^{q+r}&)(1+x+\dots+x^{p-1}) \\
&\times\sum_{i,j,k\ge0}x^{iqr+jpr+kpq}\pmod{x^{\gf(T)+1}},
\end{aligned}\end{equation}
which readily follows from \eqref{3} and \eqref{4}. In this formulation, the central role played by integers representable by linear combinations of $qr$, $pr$ and $pq$ with nonnegative coefficients is quite clear. By the Chinese remainder theorem, every integer $n$ has a unique representation in the form
\begin{equation}\label{4.2}
n=x_nqr+y_npr+z_npq+\gd_npqr,
\end{equation}
with $0\le x_n<p$, $0\le y_n<q$, $0\le z_n<r$, $\gd_n\in\Z$, and the correspondence $n\longleftrightarrow(x_n,y_n,z_n,\gd_n)$ is well defined. Note that the coefficients $x_n$, $y_n$ and $z_n$ satisfy the congruences
\begin{equation}\label{4.3}
n\equiv x_nqr\pmod p, \quad n\equiv y_npr\pmod q, \quad n\equiv z_npq\pmod r
\end{equation}
and
\begin{equation}\label{4.4}
x_n\equiv nx_1\pmod p, \quad y_n\equiv ny_1\pmod q, \quad z_n\equiv nz_1\pmod r.
\end{equation}
An integer $n$ is representable as a nonnegative linear combination of $qr$, $pr$ and $pq$ if and only if $\gd_n\ge0$. We are interested in representable integers $n<pqr$, and in this range representable comes to $\gd_n=0$. We let $\chi(n)$ be the characteristic function of such integers
\begin{equation}\label{4.5}
\chi(n)\coloneq\begin{cases}
1, &\text{if } \gd_n=0 \\
0, &\text{otherwise.} 
\end{cases}\end{equation}

Using $\chi$, we make the point of the relation \eqref{4.1} more explicit.

\begin{lemma}\label{L2}
We have,
\begin{equation}\label{4.6}
a_m=\sum_{m-p<n\le m}\bigl(\chi(n)-\chi(n-q)-\chi(n-r)+\chi(n-q-r)\bigr).
\end{equation}
\end{lemma}

\begin{proof}
This is just a more explicit form of \eqref{4.1} and it is an immediate consequence of \eqref{4.1} and \eqref{4.5}.
\end{proof}

\begin{definition}
We write $\la n\ra_m$ to denote the least nonnegative residue of $n$ modulo $m$, i.e.,
\[  \la n\ra_m\equiv n\pmod m \quad\text{and}\quad 0\le\la n\ra_m<m.  \]
\end{definition}
We will find this notation handy in what follows. For starters, we use it to rewrite the first two congruences in \eqref{4.4} as equations
\begin{equation}\label{4.7}
x_n=\la nx_1\ra_p \quad\text{and}\quad y_n=\la ny_1\ra_q.
\end{equation}
(We skip the analogue of this for $z_n$ for it will play no further role.)

Let $r^*$ be a multiplicative inverse of $r$ modulo $pq$ and observe that, by \eqref{4.2},
\begin{equation}\label{4.8}
nr^*\equiv x_nq+y_np\pmod{pq}.
\end{equation}
The quantity
\begin{equation}\label{4.9}
f(n)=f_T(n)\coloneq x_nq+y_np
\end{equation}
plays a key role in our analysis due to the characterization of representable numbers $n$ given in \refL{3} below. Note that, since $0\le f(n)<2pq$, it must be that
\begin{equation}\label{4.10}
f(n)=\la nr^*\ra_{pq} \quad\text{or}\quad f(n)=\la nr^*\ra_{pq}+pq.
\end{equation}

\begin{lemma}\label{L3}
For $n<pqr$, we have
\[  \chi(n)=1 \quad\text{if and only if}\quad f(n)=\la nr^*\ra_{pq}\le\lfloor n/r\rfloor.  \]
\end{lemma}

\begin{proof}
This follows readily from \eqref{4.2}, \eqref{4.9} and \eqref{4.10} (see \cite[(37)-(39)]{Ba.IEP}).
\end{proof}

Our next result is also contained in \cite[Theorem 3]{Ba.IEP}.

\begin{lemma}\label{L4}
We have,
\begin{equation}\label{4.11}
\CA_{\{p,q,r\}}=\CA_{\{p,q,r+pq\}}\qquad[r>\max(p,q)],
\end{equation}
and, in particular, to study $A^\pm(p,q,r)$ we may assume that $r>pq$. Furthermore, if $r'\equiv-r\pmod{pq}$ and $r'>p,q$, then
\begin{equation}\label{4.12}
\CA_{\{p,q,r'\}}=-\CA_{\{p,q,r\}}.
\end{equation}
\end{lemma}

In an earlier work Kaplan \cite{Ka.Flat} derived these identities under the assumption that $r,r'>pq$.

We end this section with two lemmas which will be used to greatly narrow down the search for coefficients $a_m$ such that $a_m=A^\pm(T)$. We use $\CN_m$ to denote the set of all arguments occuring in the sum \eqref{4.6}, that is,
\begin{equation}\label{4.13}
\CN_m\coloneq I_{m}\cup I_{m-q}\cup I_{m-r}\cup I_{m-q-r},
\end{equation}
where $I_x\coloneq(x-p,x]\cap\Z$.

\begin{lemma}\label{L5}
The identity $a_m=a_{m-pq}$ holds for every index $m$ such that the range $\CN_m$ contains no multiples of $r$. An if $\CN_m$ contains no multiples of $q$, then $a_m=a_{m-pr}$. \emph{(}Of course, we extended the definition of $a_m$ by setting $a_i=0$ for $i<0$.\emph{)}
\end{lemma}
\begin{proof}
This is Lemma 7 in \cite{Ba.IEP}.
\end{proof}

\begin{lemma}\label{L6}
If $\CN_m$ contains exactly two multiples of $r$, say $r\mid l$,
\begin{equation}\label{4.14}
l\in I_m\cup I_{m-q-r}\quad\text{and}\quad l\mp r\in I_{m-r}\cup I_{m-q},
\end{equation}
then
\begin{equation}\label{4.15}
a_m-a_{m-pq}=\chi(l)-\chi(l\mp r).
\end{equation}
If $\CN_m$ contains exactly two multiples of $q$, say $q\mid l'$,
\begin{equation}\label{4.16}
l'\in I_m\cup I_{m-q-r}\quad\text{and}\quad l'\mp q\in I_{m-q}\cup I_{m-r},
\end{equation}
then
\begin{equation}\label{4.17}
a_m-a_{m-pr}=\chi(l')-\chi(l'\mp q).
\end{equation}
The pair of choices of signs $\mp$ above is unique and corresponds to the inclusion $l\mp r, l'\mp q\in\CN_m$.
\end{lemma}
\begin{proof}
The validity of both of these assertions is contained in the proof of Lemma 8 in \cite{Ba.IEP}.
\end{proof}

\section{Proof of \refT{3}}
\subsection{Preparation and Reductions}
We are now ready to focus on the family of polynomials $Q_T$ covered by \refT{3}---polynomials with parameters $p,q,r,r'$ satisfying \eqref{6} and \eqref{7}. Noting that the congruences $rt\equiv1\pmod{pq}$ and $r't\equiv-1\pmod{pq}$ imply that $r'\equiv-r\pmod{pq}$, we immediately conclude that \eqref{8} is a consequence of \eqref{4.12} in \refL{4}. This reduces the theorem to evaluation of $A^\pm(p,q,r)$ for $p,q,r$ satisfying \eqref{6}. Furthermore, recalling that \refT{3}(i) corresponding to the case $t=1$ is a special case of \refT{4} (as discussed in Section 2), we make our second reduction leaving us to consider $t>1$---Parts (ii-iv)--- as we shall assume henceforth.

Consider the subranges $I_m$ and $I_{m-q}$ of $\CN_m$. They are disjoint ($q>p$) and if $r\ge p+q$, there cannot be more than one multiple of $r$ in $I_m\cap I_{m-q}$, whence there are at most two multiples of $r$ in $\CN_m$. Recall that, by \eqref{4.11}, we are free to assume that $r>pq$, as we do now. Therefore, any range $\CN_m$ which contains a multiple of $r$ must contain exactly two of them, as given in \eqref{4.14}. We show now that the condition $rt\equiv1\pmod{pq}$ implies that the same is true for multiples of $q$: any range $\CN_m$ which contains a multiple of $q$ must contain exactly two of them, as given in \eqref{4.16}. Indeed, if the number of multiples of $q$ in $\CN_m$ exceeded two, there would be four of them, one in each of the $I$-subranges. So consider the possibility that $l'\in I_m$ and $l'-r+i\in I_{m-r}$ are both multiples of $q$. Then $q\mid(i-r)$ for some $i$ with $0<|i|<p$. But, by \eqref{6},
\[  t(r-i)\equiv1-ti\not\equiv0\pmod q,  \]
since $1<t\le|ti|\le(p-1)^2<q-1$, a contradiction.

\begin{lemma}\label{L7}
If $m$ is the smallest index such that $a_m=A^+(T)$, or it is the smallest index such that $a_m=A^-(T)$, then the range $\CN_m$ must contain exactly two multiples of $r$ and exactly two multiples of $q$. Let us write $l$ and $l'$ for the multiples of $r$ and $q$, respectively, in the subranges:
\begin{align}
l,l'\in I_m\cup I_{m-q-r}, \qquad & \text{for } a_m=A^+(T), \label{6.1}\\
l,l'\in I_{m-r}\cup I_{m-q}, \qquad & \text{for } a_m=A^-(T). \label{6.2}
\end{align}
(Note that $l=l'$ is a possibility in either case.) Then, in either case, we have
\begin{equation}\label{6.3}
\chi(l)=1\quad\text{and}\quad \chi(l\pm r)=0 \qquad[l\pm r\in\CN_m]
\end{equation}
and
\begin{equation}\label{6.4}
\chi(l')=1\quad\text{and}\quad \chi(l'\pm q)=0 \qquad[l'\pm q\in\CN_m].
\end{equation}
\end{lemma}
\begin{proof}
In either case, $\CN_m$ must contain multiples of $r$ and $q$, by \refL{5}, whence it must contain exactly two of each. If $a_m=A^+(T)$, then, by the definition of $m$, $l$ and $l'$, and by \eqref{4.15} and \eqref{4.17}, we have
\[  a_m-a_{m-pq}=\chi(l)-\chi(l\mp r)=1  \]
and
\[  a_m-a_{m-pr}=\chi(l')-\chi(l'\mp q)=1,  \]
and \eqref{6.3} and \eqref{6.4} follow. If $a_m=A^-(T)$, we argue in the same way and write
\[  -(a_m-a_{m-pq})=\chi(l)-\chi(l\pm r)=1  \]
and
\[  -(a_m-a_{m-pr})=\chi(l')-\chi(l'\pm q)=1,  \]
to get \eqref{6.3} and \eqref{6.4} again.
\end{proof}

With \refL{7} to hand, we can announce our strategy for determining $A^\pm(T)$. We can restrict our attention to coefficients $a_m$ with ranges $\CN_m$ fulfilling the requirements \eqref{6.3} and \eqref{6.4}. Evaluating such $a_m$ will determine $A^\pm(T)$ as the maximum and minimum of these values. We can significantly streamline these calculations by combining all possible choices of signs in $A^\pm$ and in \eqref{6.3} and \eqref{6.4} into a single structure capturing all these different cases. To this end, we introduce the sum $S(Q,R;M)$ defined as follows. For $M\le\gf(T)$, we put
\begin{equation}\label{6.5}
S(Q,R;M)\coloneq\sum_{M-p<n\le M}\bigl(\chi(n)-\chi(n+Q)-\chi(n+R)+\chi(n+Q+R)\bigr),
\end{equation}
where $Q=q$ or $-q$ and $R=r$ or $-r$. Note that, by \eqref{4.6},
\[  a_m=S(-q,-r;m)=S(q,r;m-q-r)  \]
and
\[  -a_m=S(-q,r;m-r)=S(q,-r;m-q).  \]
As before \eqref{4.13}, we write $I_M=(M-p,M]\cap\Z$, but now the set of all arguments in the sum \eqref{6.5} becomes
\begin{equation}\label{6.8}
\CN_M\coloneq I_M\cup I_{M+Q}\cup I_{M+R}\cup I_{M+Q+R}.
\end{equation}
Of special interest to us are the triples $(Q,R;M)$ with the following properties. The subrange $I_M$ contains $l$, a multiple of $r$, and
\begin{equation}\label{6.9}
\chi(l)=1 \quad\text{and}\quad \chi(l+R)=0,
\end{equation}
as well as $I_M\cup I_{M+Q+R}\ni l'$, a multiple of $q$, and
\begin{equation}\label{6.10}
\chi(l')=1 \quad\text{and}\quad \chi(l'\pm Q)=0 \qquad[l'\pm Q\in\CN_M].
\end{equation}
We denote the set of all such triples $(Q,R;M)$, with $M\le\gf(T)$, by $\CM$.

\begin{lemma}\label{L8}
We have
\begin{equation}\label{6.11}
A^+(T)=\max_{\substack{(Q,R;M)\in\CM\\\sgn Q=\sgn R}}S(Q,R;M)
\end{equation}
and
\begin{equation}\label{6.12}
-A^-(T)=\max_{\substack{(Q,R;M)\in\CM\\\sgn Q=-\sgn R}}S(Q,R;M).
\end{equation}
\end{lemma}
\begin{proof}
Consider $A^+(T)$ first. To prove \eqref{6.11}, we need to show that $A^+(T)=S(Q,R;M)$ for some triple $(Q,R;M)\in\CM$ with $\sgn Q=\sgn R$. Let $m$ be the smallest index such that $a_m=A^+(T)$. Apply \refL{7} and consider the $I$-subrange containing $l$ in \eqref{6.1}. If $l\in I_m$, then \eqref{6.3} and \eqref{6.4} are equivalent to \eqref{6.9} and \eqref{6.10} with $M=m$, $Q=-q$ and $R=-r$. This shows that $A^+(T)=S(-q,-r;m)$ with $(-q,-r;m)\in\CM$, as required. In the second case $l\in I_{m-q-r}$, \eqref{6.3} and \eqref{6.4} are equivalent to \eqref{6.9} and \eqref{6.10} with $M=m-q-r$, $Q=q$ and $R=r$, so that $(q,r;m-q-r)\in\CM$, and $A^+(T)=S(q,r;m-q-r)$. This proves \eqref{6.11}.

The same argument, modulo some technical details, also proves \eqref{6.12}. Let $m$ be the smallest index such that $a_m=A^-(T)$, and apply \refL{7}. Consider the $I$-subranges containing $l$ in \eqref{6.2}. In the two cases, $l\in I_{m-r}$ and $l\in I_{m-q}$, \eqref{6.3} and \eqref{6.4} are equivalent to \eqref{6.9} and \eqref{6.10} with $(Q,R;M)=(-q,r;m-r)$ and $(Q,R;M)=(q,-r;m-q)$, respectively. In either case $(Q,R;M)\in\CM$, $-A^-(T)=S((Q,R;M)$, and \eqref{6.12} follows.
\end{proof}

\refL{8} refines the strategy for determining $A^\pm(T)$ that was indicated following \refL{7}, which now reads: evaluate $S(Q,R;M)$ for $(Q,R;M)\in\CM$ and apply \refL{8}. Unfortunately, we cannot evaluate $S(Q,R;M)$ in one go due to the fact that \eqref{6.8}-\eqref{6.10} determine three fundamentally distinct type of ranges $\CN_M$. These correspond to the three cases: (Case 1) $l=l'$; (Case 2) $l\neq l'$ and $l'\in I_M$; and (Case 3) $l'\in I_{M+Q+R}$. Each of these cases requires a separate treatment, and the natural order of these cases given above happens to coincide with their order of difficulty.

Before we start evaluating $S(Q,R;M)$ we need to fully digest and make more explicit our assumptions in \eqref{6}. To that end, we put
\begin{equation}\label{6.15}
q=\vta p+t \qquad[2\le t\le p-1,\ \vta\ge p]
\end{equation}
and, recalling that $s\in[2,p-1]$ stands for $st\equiv1\pmod p$, we write
\begin{equation}\label{6.16}
st=\eta p+1, \quad\text{so that,}\quad 0<\eta<t.
\end{equation}
Congruence $rt\equiv1\pmod{pq}$ implies that $rq\equiv1\pmod p$ and $rp\vta\equiv-1\pmod q$. This and \eqref{4.3} with $n=1$ give $x_1=1$ and $y_1=q-\vta$. Substituting these values into \eqref{4.7} yields
\begin{equation}\label{6.17}
x_n=\la n\ra_p \quad\text{and}\quad y_n=\la-\vta n\ra_q.
\end{equation}
Furthermore, \eqref{4.9} and \eqref{4.8} become
\begin{equation}\label{6.18}
f(n)=x_nq+y_np\equiv nt\pmod{pq},
\end{equation}
\eqref{4.10} becomes
\begin{equation}\label{6.19}
f(n)=\la nt\ra_{pq} \quad\text{or}\quad f(n)=\la nt\ra_{pq}+pq,
\end{equation}
and \refL{3} now reads: For $n<pqr$, we have
\begin{equation}\label{6.20}
\chi(n)=1 \quad\text{if and only if}\quad f(n)=\la nt\ra_{pq}\le\lfloor n/r\rfloor.
\end{equation}
From \eqref{6.17}, the values
\begin{equation}\label{6.21}
x_q=t,\quad y_q=0,\quad x_r=s
\end{equation}
are immediate. Parameter $y_r$ satisfies the identity
\begin{equation}\label{6.22}
x_r\vta+y_r+\eta=q.
\end{equation}
To see this, we observe that $f(r)\equiv1\pmod{pq}$, whence
\begin{equation}\label{6.23}
x_rq+y_rp=1+pq.
\end{equation}
But, by \eqref{6.15} and \eqref{6.16},
\[  x_rq+y_rp=x_r(\vta p+t)+y_rp=p(x_r\vta+y_r+\eta)+1,  \]
and \eqref{6.22} follows.

Note that \eqref{6.22} features the combination $(q-y_r)$. This combination will occur frequently enough that it will be convenient to put $y'_r=q-y_r$, so, for example, \eqref{6.22} may be written as $y'_r=x_r\vta+\eta$. More generally, we introduce the notation
\begin{equation}\label{6.24}
x'_Q\coloneq p-x\sq,\quad x'_R\coloneq p-x\sr\quad\text{and}\quad y'_R\coloneq q-y\sr.
\end{equation}
Note that $y_{\pm q}=0$, so there is no need for $y'_Q$ term. We will also find it convenient to put
\begin{equation}\label{6.25}
\ep\sq\coloneq\sgn(Q) \quad\text{and}\quad \ep\sr\coloneq\sgn(R).
\end{equation}
With this terminology in place, let us record the following extensions of \eqref{6.23} and \eqref{6.22}.

\begin{lemma}\label{L9}
We have,
\begin{equation}\label{6.26}
x\sr q+y\sr p=\ep\sr+pq
\end{equation}
and
\begin{equation}\label{6.27}
x\sr\vta+y\sr+\eta\sr=q,
\end{equation}
where $\eta_r=\eta$ and $\eta_{-r}=t-\eta$, so that $0<\eta\sr<t$, by \eqref{6.16}. Furthermore,
\begin{equation}\label{6.28}
y'_R=x\sr\vta+\eta\sr \quad\text{and}\quad y\sr=x'_R\vta+t-\eta\sr
\end{equation}
\end{lemma}
\begin{proof}
Identity \eqref{6.26} contains \eqref{6.23} and follows from the congruence $f(R)\equiv\ep\sr\pmod{pq}$. For $R=r$, \eqref{6.27} is the same as \eqref{6.22}. For $R=-r$, the left side of \eqref{6.27} is
\[
(p-x_r)\vta+q-y_r+t-\eta=2q-(x_r\vta+y_r+\eta)=q,
\]
by \eqref{6.17}, \eqref{6.15} and \eqref{6.22}.

The first identity in \eqref{6.28} is equivalent to \eqref{6.27}. To get the second, we multiply the first by $-1$ and add $q$ to both sides.
\end{proof}

Let $l\le\gf(T)$ denote an arbitrary multiple of $r$, so that $lt\equiv l/r\pmod{pq}$, and observe that, by \eqref{6.19} and \eqref{6.20} with $n=l$,
\[  \chi(l)=1 \quad\text{if and only if}\quad f(l)=l/r.  \]
To put this another way,
\begin{equation}\label{6.29}
f(l)=\begin{cases}
l/r, &\text{if } \chi(l)=1 \\
l/r+pq, &\text{if } \chi(l)=0,
\end{cases}\end{equation}
which will serve as a convenient reference. We will also find it very useful to record the following evaluation.

\begin{lemma}\label{L10}
Let $l\le\gf(T)$ be an arbitrary multiple of $r$ (so $\chi(l)=0$ or 1). Then for $0<i<p$, we have
\[  \chi(l+i)=0.  \]
\end{lemma}
\begin{proof}
By \eqref{6.20}, our claim is equivalent to the assertion that
\[  f(l+i)>\Big\lfloor\frac{l+i}r\Big\rfloor=\frac lr \qquad[0<i<p].  \]
Now, by \eqref{6.18},
\[  f(l+i)\equiv(l+i)t\equiv l/r+it\pmod{pq}.  \]
But $l/r<(p-1)(q-1)$, since $l\le\gf(T)$, and $0<it\le(p-1)^2<q$, whence the right side of this congruence is a number in $(0,pq)$. Therefore, by \eqref{6.19}, $f(l+i)\ge l/r+it$, and the claim follows.
\end{proof}

\subsection{Evaluation of $S(Q,R;M)$: (Case 1) $l=l'$}
\begin{lemma}\label{L11}
Let $(Q,R;M)\in\CM$ and assume that $l=l'$, so that $l+Q\in\CN_M$ and \eqref{6.9} and \eqref{6.10} read
\[  \chi(l)=1 \quad\text{and}\quad \chi(l+R)=\chi(l+Q)=0.  \]
Since $l$ is a multiple of $qr$, write $l=aqr$. Then $a$ must satisfy the constraints
\begin{equation}\label{7.1}
a\le x'_R-1, x'_Q-1,
\end{equation}
and the bound
\begin{equation}\label{7.2}
S(Q,R;M)\le\min(a,x\sr)+1
\end{equation}
holds. Moreover, for every integer $a$ satisfying \eqref{7.1} there is a triple $(Q,R;M)$ with $l=aqr$ such that \eqref{7.2} holds with equality. In particular, we have
\begin{equation}\label{7.3}
\max_{\substack{(Q,R;M)\in\CM\\l=l'}}S(Q,R;M)=\min(x\sr+1, x'_R, x'_Q).
\end{equation}
\end{lemma}

\begin{proof}
In addition to integer intervals such as $I_M$, we need to consider larger ranges of the form
\[  \CR_1(m)\coloneq(m-p,m+p)\cap\N,  \]
for certain special values of arguments $m$, such as $m=l$. Thus our first step is to evaluate $\chi$ on $\CR_1(l)$. In doing so we evaluate $\chi$ on every possible $I_M$ containing $l$.

We will show that, for $|i|<p$,
\begin{equation}\label{7.4}
\chi(l+i)=1 \quad\text{if and only if}\quad -a\le i\le0.
\end{equation}
By \eqref{6.20}, this is equivalent to showing that
\begin{equation}\label{7.5}
f(l+i)\le\Big\lfloor\frac{l+i}r\Big\rfloor \quad\text{if and only if}\quad -a\le i\le0.
\end{equation}
Note that the subrange corresponding to $0<i<p$ has already been dealt with in \refL{10}. Observe that $x_l=a$ and $y_l=0$, and that, by\eqref{6.17} and \eqref{6.18},
\[
f(l+i)=\la x_l+i\ra_pq+\la y_l-\vta i\ra_qp=\la a+i\ra_pq+\la-\vta i\ra_qp.
\]
But, by \eqref{6.15}, $(a+i)q-\vta ip=aq+it$, and we conclude that
\begin{equation}\label{7.6}
f(l+i)=\begin{cases}
aq+it, &\text{for } -a\le i\le0\\
aq+it+pg, &\text{for } -p<i<-a.  \end{cases}
\end{equation}
From this \eqref{7.5}/\eqref{7.4} follow since $\lfloor\frac{l+i}r\rfloor=aq+\lfloor\frac ir\rfloor$.

Next we consider the implication of the requirement $\chi(l+R)=0$. Definition of $f$, \eqref{6.18} and \eqref{6.17}, gives
\[  f(l+R)=\la a+x\sr\ra_pq+y\sr p,  \]
and \eqref{6.29}, with $l$ replaced by $l+R$, \eqref{6.25} and \eqref{6.26} give
\[  f(l+R)=aq+\ep\sr+pq=(a+x\sr)q+y\sr p.  \]
It follows that $a+x\sr<p$, establishing the first claim in \eqref{7.1}.

We are now ready to evaluate $\chi$ on $\CR_1(l+R)$. We will show that 
\begin{equation}\label{7.7}
\chi(l+R+i)=1 \quad\text{if and only if}\quad -a-x\sr\le i\le-x\sr-1.
\end{equation}
By \eqref{6.20}, this is equivalent to showing that
\begin{equation}\label{7.8}
f(l+R+i)\le\Big\lfloor\frac{l+R+i}r\Big\rfloor \quad\text{if and only if}\quad -a-x\sr\le i\le-x\sr-1.
\end{equation}
Once again, by \refL{10}, it remains to consider the subrange of $\CR_1(l+R)$ corresponding to $-p<i<0$. Here we have the following list of ingredients.
\begin{align*}
\bullet\quad &\Big\lfloor\frac{l+R+i}r\Big\rfloor=aq+\ep\sr-1, \qquad\text{by \eqref{6.25}};\\[4pt]
\bullet\quad &f(l+R+i)=\la a+x\sr+i\ra_pq+\la y\sr-\vta i\ra_qp, \qquad \text{by \eqref{6.18} and \eqref{6.17}};\\[4pt]
\bullet\quad &\la a+x\sr+i\ra_p=\begin{cases}
a+x\sr+i, &\text{for } -a-x\sr\le i<0\\
a+x\sr+i+p, &\text{for } -p<i<-a-x\sr.  \end{cases} \\[4pt]
\bullet\quad &\la y\sr-\vta i\ra_q=\begin{cases}
y\sr-\vta i, &\text{if }  y\sr-\vta i<q \\
y\sr-\vta i-q, &\text{if }  y\sr-\vta i\ge q.  \end{cases}\\[4pt]
\bullet\quad &(a+x\sr+i)q+(y\sr-\vta i)p=aq+\ep\sr+it+pq, \qquad\text{by \eqref{6.15} and \eqref{6.26}}.
\end{align*}
This list shows that the inequality $f(l+R+i)\le\lfloor\frac{l+R+i}r\rfloor$ holds if and only if
\[  f(l+R+i)=aq+\ep\sr+it,  \]
which, in turn, holds if and only if
\[  -a-x\sr\le i\le-(q-y\sr)/\vta=-y'_R/\vta=-x\sr-\eta\sr/\vta,  \]
where the last equality holds by \eqref{6.28}. We complete the proof of \eqref{7.8}/\eqref{7.7} by recalling that $\eta\sr<t<p$ (see \eqref{6.27}) and $\vta\ge p-1$, since $q>p^2$.

Next we consider the implications of the requirement $\chi(l+Q)=0$. By \eqref{6.20}, this means that
\[  f(l+Q)>\Big\lfloor\frac{l+Q}r\Big\rfloor=aq+\Big\lfloor\frac{\ep\sq}r\Big\rfloor.  \]
But $l+Q$ is a multiple of $q$, and we conclude that
\[  f(l+Q)=\la a+x\sq\ra_pq=(a+x\sq)q,  \]
that is $a+x\sq<p$. This completes the proof of \eqref{7.1}.

We are now ready to evaluate $\chi$ on $\CR_1(l+Q)$. We will show that
\begin{equation}\label{7.9}
f(l+Q+i)>\Big\lfloor\frac{l+Q+i}r\Big\rfloor \qquad[|i|<p],
\end{equation}
from which we deduce, by \eqref{6.20}, that $\chi\equiv0$ in this range. Now, since $l$ and $Q$ are multiples of $q$, we have, by \eqref{6.18} and \eqref{6.17},
\begin{equation}\label{7.10}
f(l+Q+i)=\la a+x\sq+i\ra_pq+\la-\vta i\ra_qp.
\end{equation}
Observe that, by \eqref{6.15}, we have
\begin{equation}\label{7.11}
(a+x\sq+i)q-\vta ip=(a+x\sq)q+it,
\end{equation}
and that
\[  (a+x\sq)q+it\le(p-1)q+(p-1)^2<pq,  \]
while
\[  (a+x\sq)q+it\ge aq+q-(p-1)^2>aq.  \]
Therefor, by \eqref{7.10}, \eqref{7.11} and \eqref{6.19}, we see that
\[  f(l+Q+i)\ge aq+x\sq q+it>aq.  \]
On the other hand,
\[  \Big\lfloor\frac{l+Q+i}r\Big\rfloor=aq+\Big\lfloor\frac{\ep\sq}r\Big\rfloor\le aq,  \]
and \eqref{7.9} follows.

We now address the remaining range $\CR_1(l+Q+R)$ and show that $\chi\equiv0$ here. This takes the form of showing that
\begin{equation}\label{7.12}
f(l+Q+R+i)>\Big\lfloor\frac{l+Q+R+i}r\Big\rfloor \qquad[|i|<p].
\end{equation}
Our argument here is essentially the same as for the range $\CR_1(l+Q)$ and we simply list the ingredients of the present case:
\begin{align*}
\bullet\quad &f(l+Q+R+i)=\la a+x\sq+x\sr+i\ra_pq+\la y\sr-\vta i\ra_qp; \\[4pt]
\bullet\quad &(a+x\sq+x\sr+i)q+(y\sr-\vta i)p=(a+x\sq)q+\ep\sr+it+pq; \\[4pt]
\bullet\quad &(a+x\sq)q+\ep\sr+it\le(p-1)q+1+(p-1)^2<pq; \\[4pt]
\bullet\quad &(a+x\sq)q+\ep\sr+it\ge aq+\ep\sr+q-(p-1)^2>aq+\ep\sr; \\[4pt]
\bullet\quad &f(l+Q+R+i)\ge(a+x\sq)q+\ep\sr+it>aq+\ep\sr; \\[4pt]
\bullet\quad &\Big\lfloor\frac{l+Q+R+i}r\Big\rfloor=aq+\ep\sr+\Big\lfloor\frac{\ep\sq}r\Big\rfloor\le aq+\ep\sr;\\[4pt]
\bullet\quad &\text{\eqref{7.12} now follows.}
\end{align*}

We are now ready to evaluate $S(Q,R;M)$. Since $l\in I_M$, $M$ satisfies $l\le M<l+p$ and $I_M\subset\CR_1(l)$. Therefore, the forgoing analysis applies to $S(Q,R;M)$. In particular, the inequality \eqref{7.2} follows from \eqref{6.5}, \eqref{7.4}, \eqref{7.7}, \eqref{7.9} and \eqref{7.12}. Moreover, the choice $M=l+p-a-1$ yields
\[  S(Q,R;M)=\sum_{l-a\le n\le l}\bigl(\chi(n)-\chi(n+R)\bigr)=\min(a,x\sr)+1,  \]
completing the proof of \eqref{7.2}. Finally, \eqref{7.3} follows form \eqref{7.2} and \eqref{7.1}.
\end{proof}

\subsection{Evaluation of $S(Q,R;M)$: (Case 2) $l'\in I_M,\ l'\ne l$}
\begin{lemma}\label{L12}
Let $(Q,R;M)\in\CM$ and assume that $l'\in I_M$ but $l'\ne l$, so that $l+Q\in\CN_M$ and \eqref{6.9}, \eqref{6.10} read
\[  \chi(l)=\chi(l')=1 \quad\text{and}\quad \chi(l+R)=\chi(l'+Q)=0.  \]
In this case $S(Q,R;M)\le0$, and
\[  \max_{\substack{(Q,R;M)\in\CM\\l'\in I_M,\,l\ne l'}}S(Q,R;M)=0.  \]
\end{lemma}
\begin{proof}
Since $q\mid l'$, $y_{l'}=0$. It will be convenient to temporarily put $a'=x_{l'}$. Then $l'\equiv a'qr\pmod{pq}$ and $f(l')=a'q$. Since $\chi(l')=\chi(l)=1$, \refL{10} tells us that $l'<l$. Let $j=l-l'$ ($>0$) and put $a=\la a'+j\ra_p$, so that, by \eqref{6.17}, $x_l=a$. Moreover, by \eqref{6.15},
\[  y_l=\la-\vta j\ra_q=q-\vta j=\vta(p-j)+t=\vta b+t,  \]
where $0<b=p-j<p$. Since $\chi(l)=1$, we see that
\begin{equation}\label{8.1}
l=x_lqr+y_lpr=aqr+(\vta b+t)pr
\end{equation}
and
\begin{equation}\label{8.2}
f(l)=aq+(\vta b+t)p=l/r.
\end{equation}
Note that
\begin{equation}\label{8.3}
aq+(\vta b+t)p=aq+b(q-t)+tp=(a+b)q+(p-b)t,
\end{equation}
and since this quantity is $<pq$, we conclude that
\begin{equation}\label{8.4}
0<a+b<p.
\end{equation}
Getting back to $l'$, we write $l'=l-p+b$, replace the temporary letter $a'$ by $x_{l'}=\la a-p+b\ra_p=a+b$, and note that $f(l')=(a+b)q$.

In the present case it will be more convenient to work with the ranges of the form
\[  \CR_2(m)\coloneq(m-p,m+b)\cap\N,  \]
where the parameter $b$ is determined by $l\in I_M$. In particular, note that
\[  (l-p,l+b)=(l'-b,l+b)=(l'-b,l'+p),  \]
so that $\CR_2(l)$ contains every possible $I_M$ containing $l$ given by \eqref{8.1}. Our first task is to evaluate $\chi$ on $\CR_2(l)$. We shall show that, for $l+i\in\CR_2(l)$,
\begin{equation}\label{8.5}
\chi(l+i)=1 \quad\text{if and only if}\quad i\in(-p,-p+b]\cup[-a,0].
\end{equation}
We argue just as in the preceding case covered by \refL{11} and establish the validity of the statement equivalent to \eqref{8.5}, where $\chi(l+i)=1$ is replaced by $f(l+i)\le\lfloor\frac{l+i}r\rfloor$. By \refL{10}, it only remains to consider $-p<i\le0$. In this subrange we have the following ingredients.
\begin{align*}
\bullet\quad &\Big\lfloor\frac{l+i}r\Big\rfloor=\frac lr+\Big\lfloor\frac ir\Big\rfloor.\\[4pt]
\bullet\quad &f(l+i)=\la a+i\ra_pq+\la\vta(b-i)+t\ra_qp. \\[4pt]
\bullet\quad &\la a+i\ra_p=\begin{cases}
a+i, &\text{for } -a\le i\le0\\
a+i+p, &\text{for } -p<i<-a. \end{cases} \\[4pt]
\bullet\quad &\la\vta(b-i)+t\ra_q=\begin{cases}
\vta(b-i)+t, &\text{for } -p+b<i\le0\\
\vta(b-i)+t-q, &\text{for } -p<i\le-p+b. \end{cases} \\[4pt]
\bullet\quad &(a+i)q+(\vta(b-i)+t)p=\frac lr+it, \quad\text{by \eqref{8.2} and \eqref{6.15}}.
\end{align*}
From this list and the condition \eqref{8.4}, the validity of \eqref{8.5} readily follows.

Next we consider the implications of the requirement $\chi(l+R)=0$. On the one hand, by \eqref{8.1}, we have
\begin{equation}\label{8.6}
f(l+R)=\la a+x\sr\ra_pq+\la\vta b+t+y\sr\ra_qp,
\end{equation}
and on the other, by \eqref{6.29}, \eqref{6.25} and \eqref{8.2},
\begin{equation}\label{8.7}
f(l+R)=\frac lr+\ep\sr+pq=(a+x\sr)q+(\vta b+t+y\sr)p.
\end{equation}
It follows that
\begin{equation}\label{8.8}
a+x\sr<p \quad\text{and}\quad \vta b+t+y\sr<q.
\end{equation}
Using \eqref{6.27}, we write
\begin{equation}\label{8.9}
\vta b+t+y\sr=q+\vta(b-x\sr)+t-\eta\sr.
\end{equation}
Recalling that $\vta>t>\eta\sr>0$, we see that the second inequality in \eqref{8.8} is equivalent to the condition $b<x\sr$. It is convenient to rewrite the requirements on $a$ and $b$ \eqref{8.8} in the equivalent form
\begin{equation}\label{8.10}
a+x\sr<p \quad\text{and}\quad b<x\sr.
\end{equation}

We are now ready to evaluate $\chi$ on $\CR_2(l+R)$. We shall show that, for $l+R+i\in\CR_2(l+R)$,
\begin{equation}\label{8.11}
\chi(l+R+i)=1 \quad\text{if and only if}\quad -x\sr-a\le i\le-x\sr+b.
\end{equation}
By \refL{10}, it only remains to consider the subrange corresponding to $-p<i<0$.  In this subrange, we have the following ingredients.
\begin{align*}
\bullet\quad &\Big\lfloor\frac{l+R+i}r\Big\rfloor=\frac lr+\ep\sr-1.\\[4pt]
\bullet\quad &f(l+R+i)=\la a+x\sr+i\ra_pq+\la\vta b+t+y\sr-\vta i\ra_qp, \quad\text{by \eqref{8.6}}. \\[4pt]
\bullet\quad &(a+x\sr+i)q+(\vta b+t+y\sr-\vta i)p=\frac lr+\ep\sr+it+pq, \quad\text{by \eqref{8.7}}. \\[4pt]
\bullet\quad &\la a+x\sr+i\ra_p=\begin{cases}
a+x\sr+i, &\text{if } -x\sr-a\le i<0\\
a+x\sr+i+p, &\text{if } -p<i<-x\sr-a. \end{cases}
\end{align*}
We need a bit more care for the remaining ingredient. By \eqref{8.9}, we have
\[  \la\vta b+t+y\sr-\vta i\ra_q=\la q+\vta(b-x\sr-i)+t-\eta\sr\ra_q.  \]
From this and \eqref{8.10} it is now clear that
\begin{align*}
\bullet\quad &\la\vta b+t+y\sr-\vta i\ra_q=\begin{cases}
\vta b+t+y\sr-\vta i, &\text{if } -x\sr+b<i<0\\
\vta b+t+y\sr-\vta i-q, &\text{if } -p<i\le-x\sr+b. \end{cases}
\end{align*}
From this list \eqref{8.11} follows in the usual way.

Next we consider the implication of the requirement $\chi(l'+Q)=0$. Recall that $x_{l'}=a+b$ and $f(l')=(a+b)q<\frac lr$, so that $f(l'+Q)=\la a+b+x\sq\ra_pq$. Recall also that $l'=l-p+b$, so that $\lfloor\frac{l'+Q}r\rfloor=\frac lr+\lfloor\frac{\ep\sq}r\rfloor$. By \eqref{6.20}, $f(l'+Q)>\frac lr-1$, and we conclude that
\begin{equation}\label{8.12}
a+b+x\sq<p.
\end{equation}

We are now ready to evaluate $\chi$ on $\CR_2(l+Q)$. We will show that
\begin{equation}\label{8.13}
\chi\equiv0 \quad\text{on}\quad \CR_2(l+Q).
\end{equation}
Our argument here is rather similar to our argument in the corresponding range in \refL{11} and does not require great care. Indeed, observe that, by \eqref{6.18},
\[  f(l+Q+i)\equiv l/r+x\sq q+it\pmod{pq},  \]
and that
\[  l/r+x\sq q+it\ge l/r+q-(p-1)^2>l/r,  \]
while, by \eqref{8.2}, \eqref{8.3} and \eqref{8.12},
\[  l/r+x\sq q+it=(a+b+x\sq)q+(p-b+i)t\le(p-1)q+(p-1)^2<pq,  \]
since $i<b$. Therefore, we have that
\[  f(l+Q+i)\ge l/r+x\sq q+it>l/r.  \]
But $\lfloor\frac{l+Q+i}r\rfloor=\frac lr+\lfloor\frac{\ep\sq}r\rfloor$, and \eqref{8.13} follows.

The situation in the remaining range $\CR_2(l+Q+R)$ is essentially identical to the range $\CR_2(l+Q)$. Indeed, from the preceding argument it is clear that
\[  f(l+Q+R+i)\equiv l/r+\ep\sr+x\sq q+it\pmod{pq},  \]
and that
\[  l/r+\ep\sr<l/r+\ep\sr+x\sq q+it<pq.  \]
This yields
\[  f(l+Q+R+i)>\frac lr+\ep\sr\ge\Big\lfloor\frac{l+Q+R+i}r\Big\rfloor,  \]
and we conclude that
\begin{equation}\label{8.14}
\chi\equiv0 \quad\text{on}\quad \CR_2(l+Q+R).
\end{equation}

We are now ready to evaluate $S(Q,R;M)$. As we already mentioned, the forgoing analysis is applicable since $I_M\subset\CR_2(l)$. Recall that $l-p+b=l'\in I_M$. From this it is plain, by \eqref{8.5}, \eqref{8.11}, \eqref{8.13} and \eqref{8.14}, that
\[  S(Q,R;M)\le S(Q,R;l)=0,  \]
as claimed.
\end{proof}

\subsection{Evaluation of $S(Q,R;M)$: (Case 3) $l'\in I_{M+Q+R}$}
It remains to evaluate $S(Q,R;M)$ for $(Q,R;M)\in\CM$, with $l'\in I_{M+Q+R}$. This case is more involved than the preceding cases and we split its discussion into two lemmas.

\begin{lemma}\label{L13}
Let $(Q,R;M)\in\CM$ and assume that $l'\in I_{M+Q+R}$, so that \eqref{6.9} and \eqref{6.10} read
\begin{equation}\label{9.1}
\chi(l)=\chi(l')=1 \quad\text{and}\quad \chi(l+R)=\chi(l'-Q))=0.
\end{equation}
Then $x\sq,x\sr>1$, and $l$ and $l'$ must be of the form
\begin{equation}\label{9.2}
l=aqr+(\vta b+\eta\sr-t)pr \quad\text{and}\quad l'=l+Q+R+p+b-x\sr,
\end{equation}
where $a$ and $b$ satisfy
\begin{equation}\label{9.3}
0\le a<x'_R \quad\text{and}\quad \max(1,x'_Q-a)\le b<x\sr.
\end{equation}
In the opposite direction, if a pair of integers $a$ and $b$ satisfies \eqref{9.3}, which is possible if and only if $x\sq,x\sr>1$, and $l$ and $l'$ are given by \eqref{9.2}, then $l$ and $l'$ satisfy \eqref{9.1}, and $l\in I_M$, $l'\in I_{M+Q+R}$ and $(Q,R;M)\in\CM$, for every
\begin{equation}\label{9.4}
l+p+b-x\sr\le M<l+p.
\end{equation}
\end{lemma}
\begin{proof}
Put $L=l+R$ and $L'=l'-Q$, and observe that $L,L'\in I_{M+R}$, so that $|L-L'|<p$. Note that $L\ne L'$, for otherwise $L$ is a multiple of $qr$, say $L=aqr$, and $\chi(L)=1$, a contradiction. Our first step is to show that $L'>L$. To that end we write $L'=L+B$, so that $|B|<p$, and observe that
\[
\Big\lfloor\frac{L'}r\Big\rfloor=\frac Lr+\Big\lfloor\frac Br\Big\rfloor 
\quad\text{and}\quad f(L')\not\equiv\frac Lr\pmod{pq},
\]
whence $f(L')>\frac Lr$, since $\chi(L')=0$. But
\[  f(L')=f(L+B)\equiv\frac Lr+Bt\pmod{pq},  \]
and we conclude that $f(L')\ge\frac Lr+Bt$. Now, note that:
\begin{itemize}
\item \:$f(L')=x\slp q\le pq-q$ (since $y\slp=0$), and
\item \:$L/r+Bt+pq\ge pq-(p-1)^2>pq-q$.
\end{itemize}
Evidently $f(L')=L/r+Bt>L/r$, so that $B>0$, as claimed.

Let us note two immediate important consequences of the fact that $L'>L$. The first is that, by \eqref{6.17},
\begin{equation}\label{9.5}
y\slp=\la y\sel-\vta B\ra_q=0 \implies y\sel=\vta B.
\end{equation}
The second is the relation
\begin{equation}\label{9.6}
x\slp=\la x\sel+B\ra_p=x\sel+B-p.
\end{equation}
The last equality here readily follows from \eqref{9.5} and the equations
\begin{equation}\label{9.7}
\begin{aligned}
f(L') &=\la x\sel+B\ra_pq+\la y\sel-\vta B\ra_qp=\la x\sel+B\ra_pq \\
&=\frac Lr+Bt
\end{aligned}
\end{equation}
and
\[  f(L)=x\sel q+y\sel p=\frac Lr+pq.  \]

We are now ready to start examining $l$ and its coefficients $x_l$ and $y_l$. Firstly, by our hypothesis and \eqref{4.2}, we have
\[
l=x_lqr+y_lpr \quad\text{and}\quad f(l)=x_lq+y_lp=l/r.
\]
Furthermore,
\begin{equation*}
\begin{aligned}
f(L) &=\la x_l+x\sr\ra_pq+\la y_l+y\sr\ra_qp=\frac Lr+pq\\
&=\frac lr+\ep\sr+pq=(x_l+x\sr)q+(y_l+y\sr)p,
\end{aligned}
\end{equation*}
by \eqref{6.26}. We conclude that
\begin{equation}\label{9.8}
x\sel=x_l+x\sr \quad\text{and}\quad y\sel=y_l+y\sr.
\end{equation}
Recalling that $y\sel=\vta B$ and $y\sr=q-\vta x\sr-\eta\sr$ (see \eqref{6.27}), we deduce the shape of $y_l$:
\begin{equation}\label{9.9}
\begin{aligned}
y_l &=y\sel-y\sr=\vta(B+x\sr)+\eta\sr-q\\
&=\vta(B+x\sr-p)+\eta\sr-t=\vta b+\eta\sr-t,
\end{aligned}
\end{equation}
where we put $b=B+x\sr-p \ (<x\sr)$. With $y_l$ identified and putting $a=x_l$, we now write
\[  l=aqr+(\vta b+\eta\sr-t)pr,  \]
which is the first equality in \eqref{9.2}. The second equality there is immediate since $l'=L'+Q$ and $L'=l+R+p+b-x\sr$. Let us also record the constraints on $a$ and $b$ implied by \eqref{9.8}. Firstly,
\begin{equation}\label{9.10}
a=x_l<p-x\sr=x'_R.
\end{equation}
Secondly, recall that $\eta\sr-t<0$ and $\vta-t>0$ (see \eqref{6.27} and \eqref{6.15}). This means that \eqref{9.9} requires $b\ge1$, so that
\begin{equation}\label{9.11}
1\le b<x\sr.
\end{equation}

It helps to return to the quantities $L$ and $L'$ and rewrite their properties given above in terms of the parameters $a$ and $b$. Thus \eqref{9.8} becomes
\begin{equation}\label{9.12}
x\sel=a+x\sr \quad\text{and}\quad y\sel=\vta B=\vta(p+b-x\sr),
\end{equation}
and \eqref{9.6} and \eqref{9.7} now read
\begin{equation}\label{9.13}
x\slp=a+x\sr+B-p=a+b
\end{equation}
and
\begin{equation}\label{9.14}
f(L')=(a+b)q=\frac Lr+(p+b-x\sr)t.
\end{equation}
Next we consider the implication of the hypothesis $\chi(l')=1$. From the relation
\[  l'=L'+Q=L+p+b-x\sr+Q  \]
we deduce that:
\begin{itemize}
\item \:$\lfloor l'/r\rfloor=L/r+\lfloor \ep\sq/r\rfloor$, and
\item \:$f(l')\not\equiv f(L)\pmod{pq} \implies f(l')<L/r$.
\end{itemize}
Knowing this allows us to determine $x_{l'}$ and $f(l')$ as follows (of course, $y_{l'}=0$). By \eqref{9.13},
\[  x_{l'}=\la x\slp+x\sq\ra_p=\la a+b+x\sq\ra_p,  \]
and we have the evaluation
\[  f(l')=\la a+b+x\sq\ra_pq.  \]
But, by \eqref{9.14},
\[  (a+b+x\sq)q=\frac Lr+(p+b-x\sr)t+x\sq q.  \]
Therefore, it must be that
\begin{equation}\label{9.15}
\begin{aligned}
f(l') &=\frac Lr+(p+b-x\sr)t+x\sq q-pq\\
&=\frac Lr+(p+b-x\sr)t-x'_Q q,
\end{aligned}
\end{equation}
and that
\begin{equation}\label{9.16}
x_{l'}=\la a+b+x\sq\ra_p=a+b-x'_Q.
\end{equation}
In particular, we see that the parameters $a$ and $b$ must also satisfy the requirement 
\begin{equation}\label{9.17}
a+b\ge x'_Q.
\end{equation}
Finally, combining \eqref{9.10}, \eqref{9.11} and \eqref{9.17} establishes \eqref{9.3}.

The existence of integers $a$ and $b$ satisfying \eqref{9.3}, as required by the hypothesis on $l$ and $l'$ of this case, implies that, in addition to $x\sr>1$,
\[  x\sr>x'_Q-a\ge x'_Q-x'_R+1=x_R-x_Q+1,  \]
and we see that $x_Q>1$. In the opposite direction, assume that $x_Q, x_R>1$. Now let $b$ be any integer satisfying
\[  \max(1,x_R-x_Q+1)\le b<x_R,  \]
so that $x'_Q-b<x'_R$, and take any integer $a$ in the range
\[  \max(0,x'_Q-b)\le a<x'_R.  \]
Note that every such pair of $a$ and $b$ satisfy \eqref{9.3}. Now let $l$ and $l'$ be given by the formulas in \eqref{9.2}, and let $M$ be any integer in the range \eqref{9.4}, so that $l\in I_M$ and $l'\in I_{M+Q+R}$. One readily verifies that the relevant parts of the forgoing argument are reversible showing that $l$ and $l'$ satisfy \eqref{9.1}. This completes the proof of the lemma.
\end{proof}

\begin{lemma}\label{L14}
Under the assumptions of \refL{13} and for $l$ of the form \eqref{9.2}, we have
\begin{equation}\label{9.18}
S(Q,R;M)\le\min(x\sr-b,a+1)+1,
\end{equation}
and \eqref{9.18} holds with equality for $M=l+p+b-x\sr$. Furthermore, we have
\begin{equation}\label{9.19}
\max_{\substack{(Q,R;M)\in\CM\\l\in I_M,\, l'\in I_{M+Q+R}}} S(Q,R;M)=\min(x'_R+1, x\sr, x\sq).
\end{equation}
\end{lemma}
\begin{proof}
In the present case we work with the supper ranges of the form
\[  \CR_3(m)=(m+b-x\sr, m+p)\cap\N,  \]
for $m=l, l+R, l+Q$ and $l+Q+R$, where the parameter $b$ is specified in \eqref{9.2} and \eqref{9.3}. These serve the same purpose as the supper ranges $\CR_i(m)$ in the previous cases. Indeed, recall that $I_M$ contains both $l$ and $l+p+b-x\sr$, so that $M$ must be in the range \eqref{9.4}. Therefore $\CR_3(l)\supset I_M$ for all possible values of $M$ under consideration.

For the range $\CR_3(l)$ we show that
\begin{equation}\label{9.20}
\chi(l+i)=1 \quad\text{if and only if}\quad \max(b-x\sr+1, -a)\le i\le0,
\end{equation}
by showing that
\begin{equation}\label{9.21}
f(l+i)\le\Big\lfloor\frac{l+i}r\Big\rfloor \quad\text{if and only if}\quad \max(b-x\sr+1, -a)\le i\le0.
\end{equation}
By \refL{10}, it only remains to consider $b-x\sr<i\le0$. Here is the list of ingredients we need.
\begin{align*}
\bullet\quad &\Big\lfloor\frac{l+i}r\Big\rfloor=\frac lr+\Big\lfloor\frac ir\Big\rfloor.\\[4pt]
\bullet\quad &f(l+i)=\la a+i\ra_pq+\la\vta b+\eta\sr-t-\vta i\ra_qp. \\[4pt]
\bullet\quad &\la\vta b+\eta\sr-t-\vta i\ra_q=\vta b+\eta\sr-t-\vta i, \quad\text{since }b\le b-i<x\sr. \\[4pt]
\bullet\quad &\la a+i\ra_p=\begin{cases}
a+i, &\text{for } \max(b-x\sr+1,-a)\le i\le0\\
a+i+p, &\text{for } b-x\sr<i<\max(b-x\sr+1,-a). \end{cases} \\[4pt]
\bullet\quad &(a+i)q+(\vta b+\eta\sr-t-\vta i)p=\frac lr+it.
\end{align*}
From this list the assertion \eqref{9.21}, and hence \eqref{9.20}, are immediate.

Next we show that
\begin{equation}\label{9.22}
\chi\equiv0 \quad\text{on}\quad \CR_3(l+R).
\end{equation}
\refL{10} takes care of $l+R+i$ for $0<i<p$, and we know that
\[  \chi(l+R)=\chi(L)=0 \quad\text{and}\quad f(l+R)=\frac lr+\ep\sr+pq.  \]
To evaluate
\[
f(l+R+i)=\la a+x\sr+i\ra_pq+\la\vta b+\eta\sr-t+y\sr-\vta i\ra_qp
\]
in the remaining range $b-x\sr<i<0$, we observe that, by \eqref{9.8},
\[  \la a+x\sr+i\ra_p=a+x\sr+i=x\sel+i,  \]
and that, by \eqref{9.8} and \eqref{9.9},
\[  \vta b+\eta\sr-t+y\sr=y\sel=\vta(p+b-x\sr),  \]
whence
\[
\la\vta b+\eta\sr-t+y\sr-\vta i\ra_q=\vta(p+b-x\sr-i)=y\sel-\vta i.
\]
This shows that, for $b-x\sr<i\le0$, we have
\begin{align*}
f(l+R+i) &=(x\sel+i)q+(y\sel-\vta i)p=\frac{l+R}r+pq+it \\
&>\Big\lfloor\frac{l+R+i}r\Big\rfloor,
\end{align*}
and \eqref{9.22} follows.

Next we consider the range $\CR_3(l+Q+R)$. In this range it is more convenient to work with respect to $l'$, which was determined to be of the form
\[  l'=l+Q+R+p+b-x\sr=l+R+x'_R+b+Q  \]
in \refL{13}. To that end, we note that
\[  \CR_3(l+Q+R)=(l'-p,l'+x\sr-b)\cap\N.  \]
So we wish to evaluate $f(l'+j)$, for $-p<j<x\sr-b$, and compare it to
\begin{equation}\label{9.23}
\Big\lfloor\frac{l'+j}r\Big\rfloor=\frac lr+\ep\sr+\big\lfloor\frac{\ep\sq}r\Big\rfloor.
\end{equation}
We will show that, in this range,
\begin{equation}\label{9.24}
f(l'+j)\le\Big\lfloor\frac{l'+j}r\Big\rfloor \quad\text{if and only if}\quad x'_Q-a-b\le j\le0.
\end{equation}

For the convenience of the reader, let us now recall the key attributes of $l'$ determined in \refL{13}. We have, by \eqref{9.16} and \eqref{9.15}:
\begin{itemize}
\item \:$x_{l'}=a+b-x'_Q,\ y_{l'}=0$,
\item \:$f(l')=(a+b-x'_Q)q=l/r+\ep\sr+(x'_R+b)t-x'_Qq$.
\end{itemize}
We are now ready to proceed. In preparation for comparing $f(l'+j)$ to \eqref{9.23} we observe that
\begin{equation}\label{9.25}
f(l'+j)=\la a+b-x'_Q+j\ra_pq+\la-\vta j\ra_qp
\end{equation}
and
\begin{equation}\label{9.26}
(a+b-x'_Q+j)q-\vta jp=\frac lr+\ep\sr+(x'_R+b+j)t-x'_Qq.
\end{equation}
We have,
\begin{equation}\label{9.27}
\la-\vta j\ra_q=\begin{cases}
q-\vta j, &\text{for } 0<j<x\sr-b \\
-\vta j, &\text{for } -p<j\le0.  \end{cases}
\end{equation}
Note that, by \eqref{9.3},
\[  (a+b-x'_Q)+(x\sr-b)=a-x'_R+x\sq<x\sq<p,  \]
whence
\begin{equation}\label{9.28}
\la a+b-x'_Q+j\ra_p=\begin{cases}
a+b-x'_Q+j, &\text{for } x'_Q-a-b\le j<x\sr-b \\
a+b-x'_Q+j+p, &\text{for } -p<j<x'_Q-a-b.  \end{cases}
\end{equation}
Combining \eqref{9.25}-\eqref{9.28} and \eqref{9.23} yields \eqref{9.24}.

Condition \eqref{9.24} determines, in the usual way, the evaluation of $\chi$ on $\CR_3(l+Q+R)$. However, we will find it helpful to have this evaluation with respect to $l+Q+R$ rather than $l'$, and, moreover, in the following form. We have,
\begin{equation}\label{9.29}
\chi(l+Q+R+i)=1
\end{equation}
if and only if $(p+b-x\sr)-(a+b-x'_Q)\le i\le p+b-x\sr$. This is immediate from \eqref{9.24} and \eqref{9.2}. It helps to note that, by \eqref{9.3}, $a+b<p$, so that
\[  (p+b-x\sr)-(a+b-x'_Q)>b-x\sr+x'_Q>b-x\sr.  \]

We now come to the range $\CR_3(l+Q)$. Our analysis of this range will proceed along familiar lines, but it will require somewhat greater effort than the other ranges. We shall see that there are two cases resulting in different outcomes. Namely, if $a<x'_Q$, then, for $b-x\sr<i<p$, we have
\begin{equation}\label{9.30}
\chi(l+Q+i)=1 \quad\text{if and only if}\quad x'_Q-a\le i<b,
\end{equation}
but if $a\ge x'_Q$, then
\begin{equation}\label{9.31}
\chi(l+Q+i)=1
\end{equation}
if and only if $\max(b-x\sr+1,x'_Q-a)\le i<b$ or $p+x'_Q-a\le i<p$. As usual, we will obtain these evaluations by solving the inequality
\[  f(l+Q+i)\le\Big\lfloor\frac{l+Q+i}r\Big\rfloor=\frac lr+\Big\lfloor\frac{\ep\sq}r\Big\rfloor,  \]
for the range of $i$ in question. Note that since $f(l+Q+i)\not\equiv l/r\pmod{pq}$, this inequality is equivalent to the inequality
\begin{equation}\label{9.32}
f(l+Q+i)<l/r,
\end{equation}
and we wish to show that it holds precisely for the subranges of $i$ given in \eqref{9.30} and \eqref{9.31}, according to the corresponding case.

We begin by listing basic facts of the case. We have:
\begin{gather}
x_{l+Q}=\la a+x\sq\ra_p, \quad y^{}_{l+Q}=y_l=\vta b+\eta\sr-t, \notag\\
f(l+Q+i)=\la a+x\sq+i\ra_pq+\la\vta b+\eta\sr-t-\vta i\ra_qp, \label{9.33}\\
(a+x\sq+i)q+(\vta b+\eta\sr-t-\vta i)p=l/r+x\sq q+it, \label{9.34}
\end{gather}
and
\begin{equation}\label{9.35}
\la\vta b+\eta\sr-t-\vta i\ra_q=\begin{cases}
\vta b+\eta\sr-t-\vta i+q, &\text{for } b\le i<p\\
\vta b+\eta\sr-t-\vta i, &\text{for } b-x\sr<i<b.  \end{cases}
\end{equation}
The analogue of \eqref{9.35} for $\la a+x\sq+i\ra_p$ requires us to consider the two cases mentioned above. We treat the simpler of the two cases $a<x'_Q$ first. In this case we have
\begin{equation}\label{9.36}
\la a+x\sq+i\ra_p=\begin{cases}
a+x\sq+i-p, &\text{for } x'_Q-a\le i<p\\
a+x\sq+i, &\text{for } b-x\sr<i<x'_Q-a,  \end{cases}
\end{equation}
since, by \eqref{9.3}, $a+b+x\sq\ge p$. Combining \eqref{9.33}-\eqref{9.36}, we deduce that \eqref{9.32} holds if and only if $x'_Q-a\le i<b$. This completes the proof of \eqref{9.30}.

Now suppose that $a\ge x'_Q$, so that
\[  x^{}_{l+Q}=\la a+x\sq\ra_p=a-x'_Q.  \]
In this case it is better to rewrite \eqref{9.33} and \eqref{9.34} in the form
\begin{equation}\label{9.37}
f(l+Q+i)=\la a-x'_Q+i\ra_pq+\la\vta b+\eta_R-t-\vta i\ra_qp
\end{equation}
and
\begin{equation}\label{9.38}
(a-x'_Q+i)q+(\vta b+\eta_R-t-\vta i)p=l/r-x'_Qq+it.
\end{equation}
Evaluation of the coefficient of $q$ in \eqref{9.37} now takes the form
\begin{equation}\label{9.39}
\la a-x'_Q+i\ra_p=\begin{cases}
a-x'_Q+i-p, &\text{for } p+x'_Q-a\le i<p\\
a-x'_Q+i, &\text{for } \max(b-x\sr+1,x'_Q-a)\le i<p+x'_Q-a\\  
a-x'_Q+i+p, &\text{for } b-x\sr<i<x'_Q-a, \end{cases}
\end{equation}
where, of course, the third equality above is omitted if the range $b-x\sr<i<x'_Q-a$ is empty. Now, from the congruence
\[  f(l+Q+i)\equiv l/r-x'_Qq+it \pmod{pq},  \]
which holds by \eqref{9.37} and \eqref{9.38}, and the inequalities
\[  l/r-x'_Qq+it<pq \quad\text{and}\quad -x'_Qq+it>-pq  \]
it follows that \eqref{9.32} holds if and only if
\begin{equation}\label{9.40}
f(l+Q+i)=l/r-x'_Qq+it.
\end{equation}
To deduce the range of $i$ where this equality holds we appeal to \eqref{9.37}-\eqref{9.39} and \eqref{9.35}. We also need to note that $b<p+x'_Q-a$, which follows by \eqref{9.10} and \eqref{9.11}. One now readily verifies that \eqref{9.40} holds precisely in the range of $i$ specified in \eqref{9.31}. Since this is equivalent to the assertion \eqref{9.31}, the proof of \eqref{9.31} is now complete.

We are now ready to evaluate $S(Q,R;M)$. As we pointed out at the outset of this proof, $I_M\subset\CR_3(l)$ since $M$ must satisfy \eqref{9.4}, so that the forgoing analysis applies to $S(Q,R;M)$. In particular, by \eqref{9.22}, we have
\begin{equation}\label{9.41}
S(Q,R;M)=\sum_{n\in I_M}\bigl(\chi(n)-\chi(n+Q)+\chi(n+Q+R)\bigr).
\end{equation}
Our next step is to show that the choice of $M=l+p+b-x\sr$ is optimal in the sense that 
\begin{equation}\label{9.42}
S(Q,R;M)\le S(Q,R;l+p+b-x\sr).
\end{equation}
To see this we note that for $p+b-x\sr<i<p$, we have
\[  \chi(l+i)=\chi(l+Q+R+i)=0,  \]
while $\chi(l+Q+i)$ might be 1, by \eqref{9.20}, \eqref{9.29} and \eqref{9.31}. Furthermore, for $b-x\sr<i\le0$, we have
\[  \chi(l+i)-\chi(l+Q+i)\ge0,  \]
by \eqref{9.20}, \eqref{9.30} and \eqref{9.31}. Applying these observations to the general sum \eqref{9.41} shows that \eqref{9.42} holds.

So the crux of the matter is the evaluation of $S(Q,R;l+p+b-x\sr)$. By \eqref{9.41}, we have
\begin{equation}\label{9.43}
S(Q,R;l+p+b-x\sr)=\sum_{n\in I(l)}\bigl(\chi(n)+\chi(n+Q+R)\bigr)-\sum_{n\in I(l)}\chi(n+Q),
\end{equation}
where $I(l)=(l+b-x\sr,l+b-x\sr+p]$. Applying \eqref{9.20} and \eqref{9.29} to the first of these sums gives
\begin{equation}\label{9.44}
\sum_{n\in I(l)}\bigl(\chi(n)+\chi(n+Q+R)\bigr)=\min(x\sr-b,a+1)+a+b-x'_Q+1.
\end{equation}
To evaluate the second sum on the right of \eqref{9.43} we must consider the cases $a<x'_Q$ and $a\ge x'_Q$ separately. In the former, we have, by \eqref{9.30},
\begin{equation}\label{9.45}
\sum_{n\in I(l)}\chi(n+Q)=a+b-x'_Q.
\end{equation}
In the latter case, evaluation \eqref{9.31} gives
\begin{align*}
\sum_{n\in I(l)}\chi(n+Q) &= \sum_{\max(b-x\sr+1,x'_Q-a)\le i<b}1 +\sum_{p+x'_Q-a\le i\le p+b-x\sr}1 \\
 &=\min(x\sr-1,a+b-x'_Q)+\max(0,a+b-x'_Q-x\sr+1) \\
 &=a+b-x'_Q.
\end{align*}
Thus the equation \eqref{9.45} holds in this case as well. Therefore, by \eqref{9.43}-\eqref{9.45}, we get
\begin{equation}\label{9.46}
S(Q,R;l+p+b-x\sr)=\min(x\sr-b,a+1)+1.
\end{equation}
The first claim \eqref{9.18} of the lemma now follows by \eqref{9.42} and \eqref{9.46}.

Finally, recall that we showed in \refL{13} that $(Q,R;M)\in\CM$ and $l'\in I_{M+Q+R}$ if and only if $l$ and $l'$ are given in \eqref{9.2} and $M$ satisfies \eqref{9.4}. Therefore, by \eqref{9.2}-\eqref{9.4}, \eqref{9.42} and \eqref{9.46}, we have, 
\begin{equation}\begin{aligned}\label{9.47}
\max_{\substack{(Q,R;M)\in\CM\\l\in I_M,\, l'\in I_{M+Q+R}}} S(Q,R;M)
&=\max_{l\text{ in \eqref{9.2}}} S(Q,R;l+p+b-x\sr)\\
&=\max_{a,b \text{ in \eqref{9.3}}} \min(x\sr-b,a+1)+1,
\end{aligned}\end{equation}
where the last maximum is taken over the admissible pairs $a$ and $b$ satisfying \eqref{9.3}. The shape of the constraint \eqref{9.3} makes it plain that to maximize $\min(x\sr-b,a+1)$ we should use the largest possible value of $a$, namely $a=x'_R-1$. With this choice of $a$, the lower bound on $b$ imposed by \eqref{9.3} becomes
\[  b\ge\max(1, x'_Q-x'_R+1)=\max(1,x\sr-x\sq+1).  \]
Therefore,
\begin{equation}\label{9.48}
\max_{a,b \text{ in \eqref{9.3}}} \min(x\sr-b,a+1)=\min(x'_R, x\sr-1, x\sq-1).
\end{equation}
Substituting \eqref{9.48} into \eqref{9.47} gives \eqref{9.19} and completes the proof of the lemma.
\end{proof}

\subsection{Completion}
Recall that we choose to treat Part (i) of the theorem as a special case of \refT{4} and that we are now assuming that $t>1$. For any of the four possible choices of the pair $Q$ and $R$, Lemmas 19, 20 and 22 tell us that
\begin{equation}\label{10.1}
\max_{(Q,R;M)\in\CM}S(Q,R;M)=\max\bigl[\min(x\sr+1,x'_R,x'_Q),\min(x'_R+1,x\sr,x\sq)\bigr].
\end{equation}
Recall the definition of $\ut$ and $\us$ and observe that, by \eqref{6.21} and \eqref{6.24},
\[  \ut=\min(x\sq,x'_Q) \quad\text{and}\quad \us=\min(x\sr,x'_R),  \]
for every choice of $Q$ and $R$. Suppose now that $\us<\ut$. Then $x\sq,x'_Q>\min(x\sr,x'_R)$ and \eqref{10.1} gives
\[
\max_{(Q,R;M)\in\CM}S(Q,R;M)=\max\bigl[\min(x\sr+1,x'_R),\min(x'_R+1,x\sr)\bigr]=s+1
\]
($Q$ and $R$ arbitrary). Part (ii) of the theorem now follows by \refL{8}.

Next, observe that, by \eqref{6.24}, $x'_Q=x_{-Q}$ and $x'_R=x_{-R}$, whence
\[  \min(x\sr+1, x'_R, x'_Q)=\min(x'_{-R}+1, x^{}_{-R}, x^{}_{-Q})  \]
and
\[  \min(x'_R+1, x\sr, x\sq)=\min(x^{}_{-R}+1, x'_{-R}, x'_{-Q}).  \]
Using these identities in \eqref{10.1} shows that, for any specific pair of $Q$ and $R$, we have
\begin{equation}\label{10.2}
\max_{(Q,R;M)\in\CM}S(Q,R;M)=\max_{(Q,R;M)\in\CM}S(-Q,-R;M).
\end{equation}
This identity is handy for the case $\us\ge\ut$ which we now consider. We make a specific choice of $Q$ and $R$ according to the value of $\ut$, namely
\begin{equation}\label{10.3}
\bigl(Q,R\bigr)=\begin{cases}
(q,r), &\text{if } \ut=t\\
(-q,-r), &\text{if } \ut=p-t.  \end{cases}
\end{equation}
Note that for this choice of $Q$ and $R$, we have
\[  x\sq=\ut \quad\text{and}\quad x'_Q\ge x\sr,x'_R.  \]
Then, by \eqref{6.11}, \eqref{10.2} and \eqref{10.1} with $Q$ and $R$ given by \eqref{10.3}, we have
\[  A^+(T)=\max_{(Q,R;M)\in\CM}S(Q,R;M)=\min(x\sr+1,x'_R).  \]
Furthermore, in exactly the same way we obtain, by \eqref{6.12},
\[  -A^-(T)=\max_{(Q,R;M)\in\CM}S(Q,-R;M)=\min(x^{}_{-R}+1,x'_{-R})=\min(x'_R+1,x\sr).  \]
Using the values of $x\sr$, $\us$ or $p-\us$, in these equations yields Parts (iii) and (iv) and completes the proof of the theorem.

\end{document}